\documentclass[12pt]{amsart}
\usepackage[top=1in, bottom=1in, left=1in, right=1in]{geometry}
\usepackage{amsfonts}
\usepackage{amsmath}
\usepackage{amssymb}
\usepackage{mathrsfs}
\numberwithin{equation}{section}
\usepackage{times}

\usepackage{siunitx}
\usepackage[usenames,dvipsnames]{color}
\usepackage{xcolor}
\usepackage{mathtools}
\usepackage{bm}
\usepackage{esvect}
\usepackage{hyperref}
\hypersetup{
    colorlinks = true,
linkcolor={black},
urlcolor={blue},
citecolor={blue},    
urlcolor = {blue},
citebordercolor = {0.33 .58 0.33},
 linkbordercolor = {0.99 .28 0.23},
 breaklinks=true}

\newcommand{\kommentar}[1]{}

\newcommand{\R}{\mathbb{R}}

\newcommand{\N}{\mathbb{N}}
\newcommand{\Z}{\mathbb{Z}}

\newtheorem{thm}{Theorem}[section]
\newtheorem{defn}[thm]{Definition}
\newtheorem{prop}[thm]{Proposition}
\newtheorem{lem}[thm]{Lemma}

\newcommand{\pn}{\sqrt{P_{n}}}

\newcommand{\seqnum}[1]{\href{https://oeis.org/#1}{\rm \underline{#1}}}

\theoremstyle{remark}

\usepackage{geometry}
\title{An Equidistribution Result for Differences Associated with Square Pyramidal Numbers}

\author{Anji Dong, Katerina Saettone, Kendra Song and Alexandru Zaharescu}
\address{
Anji Dong: Department of Mathematics,
University of Illinois Urbana-Champaign,
Altgeld Hall, 1409 West Green Street,
Urbana, IL, 61801, USA}
\email{anjid2@illinois.edu}

\address{
Katerina Saettone: Department of Mathematics,
University of Illinois Urbana-Champaign,
Altgeld Hall, 1409 West Green Street,
Urbana, IL, 61801, USA}
\email{kas18@illinois.edu}

\address{
Kendra Song: Department of Mathematics,
University of Illinois Urbana-Champaign,
Altgeld Hall, 1409 West Green Street,
Urbana, IL, 61801, USA}
\email{kendras4@illinois.edu}

\address{
Alexandru Zaharescu: Department of Mathematics,
University of Illinois Urbana-Champaign,
Altgeld Hall, 1409 West Green Street,
Urbana, IL, 61801, USA and Simion Stoilow Institute of Mathematics of the Romanian Academy, 
P.\ O.\ Box 1-764, RO-014700 Bucharest, Romania}
\email{zaharesc@illinois.edu} 
\begin{document}
\begin{abstract}
    We provide an asymptotic formula for the average value of the sequence \seqnum{A351830}: $a_{n} = |P_{n} - y^{2}_{n}|$ for $1 \leq n \leq x$, where $P_{n}$ is the $n$-th square pyramidal number and $y^{2}_{n}$ is the closest square to $P_{n}$. Moreover, we supply asymptotic formulas for the $k$-th moment of the same sequence, for any fixed natural number $k$.
\end{abstract}
\maketitle

\section{Introduction}\label{introduction}

In the 1500s, Sir Walter Raleigh proposed the problem of counting the cannonballs in a square pyramid during a sea voyage, which is now known as the Cannonball problem. Later, the problem was officially proposed by Lucas \cite{questions}, which studies what integers are both a square and a square pyramidal number. It was eventually proved by Watson \cite{WatsonCannonball} that the only solutions are $0,1,$ and $4900$. 

There are many integer sequences associated with the Cannonball problem, for example, see \seqnum{A001032}, \seqnum{A351830}, \seqnum{A350886}, \seqnum{A350887}, and \seqnum{A350888} in the On-Line Encyclopedia of Integer Sequences (OEIS) \cite{oeis}. In this paper, we establish additional properties of the sequence \seqnum{A351830}. This is given by
\begin{align}\label{a_n first definition}
    a_{n} = \big\lvert P_{n} - y^{2}_{n} \big\rvert,
\end{align}
where $P_{n}$ is the $n$-th square pyramidal number defined as
\begin{align}
   P_n &= \sum_{i = 1}^n i^2= \frac{2n^3+3n^2+n}{6},\label{defn: square pyramidal number} 
\end{align} 
and $y^{2}_{n}$ is the closest square to $P_{n}$. In fact, Paolo Xausa \cite{paolo} computed the first ten thousand elements of the sequence.

Note that the only values of $n$ such that $a_{n} = 0$ are $0, 1,$ and $24$, which correspond to the solutions to the Cannonball problem. Moreover, noted by Conway and Sloane \cite{conway1982, conway1999}, the solution corresponding to $n=24$ allows for the Lorentzian construction of the Leech lattice. Further applications and generalizations of the Cannonball problem are discussed by Laub \cite{Laub1990}, Beeckmans \cite{beeckmans1994}, and Bennett \cite{bennett2002}. 

For real $x \geq 1$, let $A(x)$ represent the average value of $(a_{n})$ for $1 \leq n \leq x$. More precisely,
\begin{align}\label{A(x) first definition}
A(x) := \frac{1}{x}\sum_{1 \leq n \leq x} a_{n}.
\end{align}

Hence, since $a_n$ represents the distance between the $n$-th square pyramidal number and the closest square, $A(x)$ represents the average distance from a square pyramidal number to its closest square up to the $\lfloor x\rfloor$-th square pyramidal number. 
Our primary goal is to provide an asymptotic formula for $A(x)$.
\begin{thm}\label{asymptotic for A(x) theorem}
For real $x\geq 1$, the function $A(x)$, given by \eqref{A(x) first definition}, satisfies the asymptotic formula
\begin{align}\label{A(x) asymptotic}
    A(x) = \frac{1}{5\sqrt{3}}x^{3/2}+O(x^{17/12}),
\end{align}
where the constant implied by the big $O$ symbol is absolute.
\end{thm}

Next, we generalize Theorem \ref{asymptotic for A(x) theorem} by considering higher moments of the sequence $(a_{n})$. For each $k\in\N$, let $M_{k}(x)$ represent the $k$-th moment of $(a_{n})_{n\leq x}$, that is,
\begin{align}\label{k moment definition}
    M_{k}(x) := \sum_{1 \leq n \leq x}a_{n}^{k}.
\end{align}
By Theorem \ref{asymptotic for A(x) theorem},
\[
M_{1}(x) = xA(x) = \frac{1}{5\sqrt{3}}x^{5/2}+O(x^{29/12}).
\]
More generally, we have the following result. 
\begin{thm}\label{Moments theorem}
For any positive integer $k$,
\begin{align}
    M_{k}(x)=  \frac{x^{\frac{3}{2}k + 1}}{3^{k/2}(\frac{3}{2}k + 1)(k+1)}+O_k\left(x^{\frac{3}{2}k+\frac{11}{12}}\right),
\end{align}
 where $M_{k}$ is defined in \eqref{k moment definition}, and the constant implied by the $O_k$ symbol depends at most on $k$.  
\end{thm}
\subsection*{Structure of the Paper} The paper is organized as follows. We begin by introducing some standard notation in Section \ref{sec: general notation}. In Section \ref{sec: auxiliary results}, we provide some auxiliary results needed in the proofs of the main theorems. Lastly, in Section \ref{sec: proofs}, we complete the proofs of Theorems \ref{asymptotic for A(x) theorem} and \ref{Moments theorem}.

\section{General Notation}\label{sec: general notation}
We employ some standard notation used throughout the article.
\begin{itemize} 
\item The expressions $f(X)=O(g(X))$, $f(X) \ll g(X)$, and $g(X) \gg f(X)$ are equivalent to the statement that $|f(X)| \leq C|g(X)|$ for all sufficiently large $X$, where $C>0$ is an absolute constant. A subscript of the form $\ll_{\alpha}$ means that the implied constant may depend on the parameter $\alpha$. Dependence on several parameters is indicated similarly, as in $\ll_{\alpha, \lambda}$.
\item Given $\alpha \in \mathbb{R}$, the notation $\|\alpha\|$ denotes the smallest distance of $\alpha$ to an integer.
\item Given a set $S$, the notation $|S|$ and $\#S$ both stand for the cardinality of $S$. 
\item For $x\in\R$, the notation $\lceil x\rceil$ denotes the ceiling of $x$, which is the smallest integer larger than $x$. The notation $\lfloor x\rfloor$ denotes the floor of $x$, which is the largest integer smaller than $x$. 
\item For real $x\in\R$, $\{x\}$ denotes the fractional part of $x$, that is, $x-\lfloor x\rfloor$.
\item The notation $e(x)$ stands for $\exp(2\pi i x)$.
\end{itemize}

\section{Auxiliary Propositions}\label{sec: auxiliary results}

In this section, we aim to show that $(a_n)$, defined in \eqref{a_n first definition}, is equidistributed in $[0,\frac{1}{2}]$. We begin with the following lemma by Kuipers and Niederreiter \cite{UniformDistributionofSequences}.
\begin{lem}[{\cite[Thm.\ 2.7, p.\ 17]{UniformDistributionofSequences}}]\label{bound for exponential sums}
    Let $a$ and $b$ be integers with $a < b$. Assume there exists some $\rho\in\R^{+}$ such that $f$ is twice differentiable on $[a,b]$ with $f''(x) \geq \rho >0$ or $f''(x) \leq -\rho < 0$ for all $x \in [a,b]$. Then,
    \begin{align}\label{bound for exponential sums inequality}
        \bigg| \sum^{b}_{n=a} e(f(n)) \bigg| \leq \bigg(\big|f'(b) - f'(a)\big| + 2\bigg) \bigg( \frac{4}{\sqrt{\rho}} + 3 \bigg).
    \end{align}
\end{lem}
\begin{lem}[{\cite[Weyl's Criterion, p.\ 1]{montgomery1994ten}}]\label{weyl criterion}
A sequence $(a_n)$ is uniformly distributed if and only if for each integer $k\neq 0$, 
\[
\lim_{N\rightarrow\infty}\frac{1}{N}\sum_{n=1}^N e(ka_n) = 0.
\]
\end{lem}
\begin{prop}\label{sqrt pn u.d}
    The sequence $\{\sqrt{P_{n}}\}$ is uniformly distributed in the unit interval $[0,1]$.
\end{prop}
\begin{proof}
Let $h(x) = \sqrt{P_x}$. It is clear that 
\begin{align*}
    h'(x) &= \frac{2x(x+1)+(2x+1)(x+1)+x(2x+1)}{2\sqrt{6}\sqrt{x(x+1)(2x+1)}}, \\
    h''(x) &= \frac{8x+4(x+1)+2}{2\sqrt{6}\sqrt{x(x+1)(2x+1)}}-\frac{(2x(x+1)+(2x+1)(x+1)+x(2x+1))^2}{4\sqrt{6}(x(x+1)(2x+1))^{3/2}}.
\end{align*} 
Let $a \in \Z\setminus\{0\}$. Since $h''$ is monotonically decreasing, take $\rho = h''(N)$. Then Lemma \ref{bound for exponential sums} gives us 
\begin{align*}
     \lim_{N \to \infty}\frac{1}{N} \bigg| \sum_{n = 1}^{N} e(ah(n)) \bigg| &\leq\  \lim_{N\to\infty}\frac{1}{N}(|ah'(N)-ah'(1)| + 2)\bigg(\frac{4}{\sqrt{ah''(N)}}+3\bigg) \\
    &= 0.
\end{align*}
Upon using Lemma \ref{weyl criterion}, we see that $\sqrt{P_{n}}$ is uniformly distributed. Therefore, the sequence $(\sqrt{P_{n}})$ is uniformly distributed $\bmod \hspace{0.1cm} 1,$ that is, the sequence of the fractional parts, $\{ \sqrt{P_{n}} \}$, is uniformly distributed in the interval $[0,1]$.
\end{proof}
\begin{prop}\label{sqrt pn - yn uniform distribution}
    The sequence $(a_n)$, defined as $|\pn - y_{n}|$, is uniformly distributed in the interval $[0, 1/2]$. 
\end{prop}
\begin{proof}
    Note that $(\lfloor \sqrt{P_n}\rfloor)^2\leq P_n\leq(\lceil\sqrt{P_n}\rceil)^2$, and since $\lceil\sqrt{P_n}\rceil-\lfloor \sqrt{P_n}\rfloor\leq 1$, it follows that $y_{n}$ is either $\lfloor \sqrt{P_{n}} \rfloor$ or $\lceil \sqrt{P_{n}} \rceil$. Therefore, $|\pn - y_{n}|$ is at most $1/2$.
    
    Consider the following set
  \begin{align}\label{exception set}
        E(x) = \left\{ n \colon 1 \leq n \leq x, |\sqrt{P_n}-y_n| \neq \|\sqrt{P_n}\|\right\},
    \end{align}
    which we refer to as the exceptional set depending on $x$. For all $n \in E(x)$, the number $n$ must satisfy one of the two following cases:
    
    \textbf{Case 1:} $y_n=\lfloor \sqrt{P_n}\rfloor$ and $\|\sqrt{P_n}\| = |\sqrt{P_n} -(y_n+1)| < |\sqrt{P_n}-y_n|$.
    
    By definition of $y_n$ in \eqref{a_n first definition}, we have   
    \begin{align}\label{case 1 step 1}
        |P_n-y_n^2| &< |P_n-(y_n+1)^2|.
    \end{align}
    Since $y_n=\lfloor \sqrt{P_n}\rfloor$, \eqref{case 1 step 1} is equivalent to
        \begin{align*}
        P_n-y_n^2 &< (y_n+1)^2-P_n, \end{align*}
        so $\sqrt{P_n} < \sqrt{y_n^2 + y_n +1/2}.$ Moreover, under the conditions of Case 1, we also have $\sqrt{P_n} > y_n + \frac{1}{2}$. These two inequalities yield 
        \begin{align*}
        y_n + \frac{1}{2} < \sqrt{P_n} <& \ \sqrt{y_n^2 + y_n + 1/2} \\
        <& \bigg(y_n + \frac{1}{2}\bigg)\bigg(1 + \frac{1}{8(y_n + \frac{1}{2})^2}\bigg).
        \end{align*}
        Therefore,
        \begin{align}
        \frac{1}{2} < \{\sqrt{P_n}\} < \ \frac{1}{2} + \frac{1}{\sqrt{P_n}}.
        \end{align}
        \textbf{Case 2:} $y_n=\lceil \sqrt{P_n}\rceil, \ \|\sqrt{P_n}\| = |\sqrt{P_n} - (y_n-1)| < |\sqrt{P_n}-y_n|$.
        
        Through arguments analogous to those in Case 1, we obtain the bounds:
        \begin{align*}
            \sqrt{y_n^2 - y_n + 1/2}<\sqrt{P_n}<y_n-\frac{1}{2},
        \end{align*}
        which again leads to
        \begin{align}\label{root Pn upper lower bound}
            \frac{1}{2} - \frac{1}{\sqrt{P_n}}< \{\sqrt{P_n}\} < \frac{1}{2}.
        \end{align}
        To bound $|E(x)|$, we first break the interval $[1,x]$ into dyadic pieces, and bound the number of $n$ in $[1,\sqrt{x}]$ trivially by $O(\sqrt{x})$. Now, for $\sqrt{x}<n<x$, using the closed form of $P_{n}$, we obtain
\begin{align}
     \sqrt{P_{n}}
     &= 2\sqrt{\frac{n\bigl(n+1\bigr)\bigl(2n+1\bigr)}{6}}  \notag\\ 
     &= \frac{1}{\sqrt{3}} n^{3/2} \sqrt{\bigg(1 + \frac{1}{n}\bigg)\bigg( 1 + \frac{1}{2n}\bigg)}  \notag\\
    &=\frac{1}{\sqrt{3}}n^{3/2}+\frac{\sqrt{3}}{4}n^{1/2}+O\bigl( n^{-1/2}\bigr).\label{closed form of P_n}
\end{align}
Combining \eqref{root Pn upper lower bound} and \eqref{closed form of P_n}, we achieve that 
\[
\bigg\lvert\{\sqrt{P_n}\} - \frac{1}{2}\bigg\rvert\leq x^{-3/4},
\]
for sufficiently large $x$. Thus, 
\begin{align}
    \lvert E(x)\rvert = \#\left\{ n \colon 1 \leq n \leq x, \left\lvert\{\sqrt{P_n}\} - \frac{1}{2}\right\rvert\leq x^{-3/4}\right\}+O(\sqrt{x}).\label{size of exceptional set}
\end{align}
Taking into account all of the above, for all $0\leq\alpha<\beta\leq 1/2$, we have
\begin{align*}
    &\lim_{x\rightarrow\infty} \frac{\#\{n \colon \|\sqrt{P_n}\|\in[\alpha,\beta]\}}{x}-\lim_{x\rightarrow\infty} \frac{|E(x)|}{x}\leq\lim_{x\rightarrow\infty} \frac{\#\{n \colon 1 \leq n \leq x, |\sqrt{P_n}-y_n|\in[\alpha,\beta]\}}{x}
\end{align*}
and
\begin{align*}
    \lim_{x\rightarrow\infty} \frac{\#\{n \colon 1 \leq n \leq x, |\sqrt{P_n}-y_n|\in[\alpha,\beta]\}}{x}\leq\lim_{x\rightarrow\infty} \frac{\#\{n \colon \|\sqrt{P_n}\|\in[\alpha,\beta]\}}{x}+\lim_{x\rightarrow\infty} \frac{|E(x)|}{x}.
\end{align*}
Applying Proposition \ref{sqrt pn u.d}, 
\begin{align*}
    \lim_{x\rightarrow\infty} \frac{\#\{n \colon \|\sqrt{P_n}\|\in[\alpha,\beta]\}}{x}&=\lim_{x\rightarrow\infty} \frac{\#\{n \colon \{\sqrt{P_n}\}\in[\alpha,\beta]\cup[1-\beta,1-\alpha]\}}{x}\notag\\
    &=2(\beta-\alpha).
\end{align*}
Moreover,
\begin{align*}
    \lim_{x\rightarrow\infty} \frac{|E(x)|}{x}&= \lim_{x\rightarrow\infty} \frac{\#\{n \colon \{\sqrt{P_n}\}\in[\frac{1}{2}-x^{-3/4},\frac{1}{2}+x^{-3/4}]\}}{x}+x^{-1/2}\\
    &=\lim_{x\rightarrow\infty} 2x^{-3/4}+x^{-1/2}=0.
\end{align*}
Thus,
\[
\lim_{x\rightarrow\infty} \frac{\#\{n \colon 1 \leq n \leq x, |\sqrt{P_n}-y_n|\in[\alpha,\beta]\}}{x} = 2(\beta-\alpha),
\]
which, by the definition of equidistribution, completes the proof of Proposition \ref{sqrt pn - yn uniform distribution}.
\end{proof}

\section{Proof of Theorems \ref{asymptotic for A(x) theorem} and \ref{Moments theorem}}\label{sec: proofs}
In this section, we provide proofs for Theorems \ref{asymptotic for A(x) theorem} and \ref{Moments theorem}. Before proving the theorems, we first introduce two useful lemmas. The following lemma is a modified version of an auxiliary lemma by Das, Robles, Dirk, and the fourth author \cite{semiprimes}. The proof is essentially the same as the original one, so we omit it here.
\begin{lem}[{\cite[Lemma~5.1]{semiprimes}}] 
\label{lem:min_max_for_minor}
Let $i\in\N$, and $F, G_1,G_2,\cdots,G_i$ be continuous, real-valued functions on $\R_+$ such that $F$ is strictly decreasing and $G_i$'s are increasing. 
Further, suppose that 
\begin{align}
\lim_{x\to\infty} F(x) = \lim_{x\to 0} G_i(x) = 0
\ \text{ and } \
\lim_{x\to0} F(x) = \lim_{x\to\infty } G_i(x) = \infty.
\end{align} 
Set $G(x):=\max\{G_1(x),\cdots, G_i(x)\}$ and $H(x):=\max\{F(x),G(x)\}$.
We then have
\begin{align}
\min_{x\in(0,\infty)} H(x) 
=
F(\min\{M_1,\cdots,M_i\})
=
G(\min\{M_1,\cdots,M_i\}),
\label{eq:lem:min_max_for_minor}
\end{align}
where $M_j$ is the solution of the equation $F(M_j) = G_j(M_j)$ for $j=1,2,\cdots,i$.\\
\end{lem}
\begin{defn}\label{Discrepancy def}
Let $(u_n)$ be a sequence of points in the circle group $\mathbb{T} = \R/\Z$.  For $0\leq \alpha\leq1$, 
\[Z(N;\alpha) = \#\{n\in \mathbb{Z} : 1 \leq n \leq N, \ 0\leq u_n\leq \alpha \bmod 1\} 
\]
Let $D(N;\alpha) = Z(N;\alpha)-N\alpha$.  Then the \textbf{discrepancy} of the sequence is 
\[
D(N) = \sup_{\alpha \in [0,1]} |D(N;\alpha)|.
\]
\end{defn}
\begin{lem}
[Erd\"{o}s-Tur\'{a}n inequality, {\cite[Coro.\ 1.1, p.\ 8]{montgomery1994ten}}]\label{erdos turan inequality}
Let $(x_n)_{n\leq N}$ be a finite sequence with discrepancy $D(N)$. Then, for any positive integer $K$, 

\[
D(N) \leq \frac{N}{K + 1} + 3\sum_{m=1}^K\frac{1}{m} \bigg\lvert\sum_{n\leq N} e(mx_n)\bigg\rvert.
\]
\end{lem}

Since the result of Theorem \ref{asymptotic for A(x) theorem} follows immediately from Theorem \ref{Moments theorem} by taking $k=1$, we begin with the proof of Theorem \ref{Moments theorem}.
\begin{proof}[Proof of Theorem \ref{Moments theorem}]
    We write
\begin{align}\label{factored A(x)}
    a_{n} = \big\lvert \sqrt{P_{n}} - y_{n}\big\rvert \big\lvert \sqrt{P_{n}} + y_{n} \big\rvert.
\end{align}
 Then, 
    \begin{align*}
        M_{k}(x) = \sum_{1\leq n\leq x}a_n^k = \sum_{1 \leq n \leq x}|\sqrt{P_n}-y_n|^k|\sqrt{P_n}+y_n|^k.
    \end{align*}
By the discussion in the proof of Proposition \ref{sqrt pn - yn uniform distribution}, $y_{n}$ is either $\lfloor \sqrt{P_{n}} \rfloor$ or $\lceil \sqrt{P_{n}} \rceil$, so
\begin{align}\label{fractional part asymptotic}
    y_{n} = \sqrt{P_{n}} + O(1),
\end{align}
and thus $y_{n} + \sqrt{P_{n}} = 2 \sqrt{P_{n}} + O(1)$. Moreover,
\begin{align}
    \bigl(\sqrt{P_n}+y_n\bigr)^k &= \bigl(2\sqrt{P_n}+O(1)\bigr)^k\notag \\
    &= {\binom{k}{0}} \bigl( 2 \sqrt{P_{n}} \bigr)^{k} + O\bigg({\binom{k}{1}} \bigl( 2 \sqrt{P_{n}} \bigr)^{k-1} \bigg)\notag \\
    &=  \bigl(2\pn \bigr)^{k} + O_k\bigl( n^{\frac{3}{2}(k-1)} \bigr).\label{preliminary sqrt pn + yn kth power asymptotic}
\end{align}
Upon using the closed form of $\sqrt{P_{n}}$ in \eqref{closed form of P_n}, we have 
\begin{align}
    \bigl(  \pn + y_{n} \bigr)^{k} &= \bigg(2 \sqrt{\frac{n(n+1)(2n+1)}{6}} \bigg)^{k} + O_k\bigl( n^{\frac{3}{2}(k-1)} \bigr) \notag\\
    &= \bigg(\frac{2}{\sqrt{3}} n^{3/2} \sqrt{\bigg(1 + \frac{1}{n}\bigg)\bigg( 1 + \frac{1}{2n}\bigg)} \bigg)^{k} + O_k\bigl( n^{\frac{3}{2}(k-1)} \bigr) \notag \\
    &= \bigg( \frac{2}{\sqrt{3}} \bigg)^{k} n^{\frac{3}{2}k} \bigg(1 + O\bigg(\frac{1}{n}\bigg) \bigg)^{k} + O_k\bigl( n^{\frac{3}{2}(k-1)} \bigr) \notag \\
    &= \bigg( \frac{2}{\sqrt{3}} \bigg)^{k} n^{\frac{3}{2}k}  + O_k\bigl( n^{\frac{3}{2}k-1} \bigr). \label{sqrt pn + yn kth power asymptotic}
\end{align}
Using \eqref{sqrt pn + yn kth power asymptotic}, upon trivially bounding $|\sqrt{P_n}-y_n|\leq 1$ in the error term, we now have that 
\begin{align*}
    M_k(x) = \sum_{1\leq n\leq x}|\sqrt{P_n}-y_n|^k\bigg( \frac{2}{\sqrt{3}} \bigg)^{k} n^{\frac{3}{2}k}  + O_k\bigl( x^{\frac{3}{2}k} \bigr).
\end{align*}
 By Proposition \ref{sqrt pn - yn uniform distribution}, choose an even $L\in \N$, which will be optimized later, we have
\begin{align*}
    M_k(x) &= \bigg( \frac{2}{\sqrt{3}} \bigg)^{k}\sum_{j=1}^{L/2} \sum_{\substack{1\leq n\leq x\\\frac{j-1}{L}<|\sqrt{P_n}-y_n|\leq\frac{j}{L}}}|\sqrt{P_n}-y_n|^k n^{\frac{3}{2}k} + O_k\bigl( x^{\frac{3}{2}k} \bigr).
\end{align*}
Note that 
\begin{align}\label{sandwich}
U_k(x,L) \leq M_k(x)\leq V_k(x,L),
\end{align}
where
\begin{align}\label{Uk(x,L)}
    U_k(x,L) = \bigg( \frac{2}{\sqrt{3}} \bigg)^{k}\sum_{j=1}^{L/2}\bigg(\frac{j-1}{L}\bigg)^k \sum_{\substack{1\leq n\leq x\\\frac{j-1}{L}<|\sqrt{P_n}-y_n|\leq\frac{j}{L}}} n^{\frac{3}{2}k} + O_k\bigl( x^{\frac{3}{2}k} \bigr),
\end{align}
and
\begin{align}\label{Vk(x,L)}
    V_k(x,L) = \bigg( \frac{2}{\sqrt{3}} \bigg)^{k}\sum_{j=1}^{L/2} \bigg(\frac{j}{L}\bigg)^k\sum_{\substack{1\leq n\leq x\\\frac{j-1}{L}<|\sqrt{P_n}-y_n|\leq\frac{j}{L}}} n^{\frac{3}{2}k} + O_k\bigl( x^{\frac{3}{2}k} \bigr).
\end{align}
We remark that one can obtain the main term of $U_k(x,L)$ and $V_k(x,L)$ at this point by use of the fact that $|\sqrt{P_n}-y_n|$ is equidistributed in $[0,1/2]$ from Proposition \ref{sqrt pn - yn uniform distribution}. However, we want a more accurate error term, so we proceed as follows. To evaluate $V_k(x,L)$, we focus on the inner sum on the right side of \eqref{Vk(x,L)}. Let $M\leq x$ be in $\N$, and define 
\begin{align}\label{S j}
S_j := \sum_{\substack{1\leq n\leq x\\\frac{j-1}{L}<|\sqrt{P_n}-y_n|\leq\frac{j}{L}}} n^{\frac{3}{2}k}  = \sum_{0\leq \ell \leq M-1}S_{j,\ell, M}
\end{align}
where 
\begin{align}\label{S(j,ell)}
    S_{j,\ell, M} := \sum_{\substack{\frac{\ell x}{M} \leq n < \frac{(\ell+1)x}{M}\\\frac{j-1}{L} \leq |\sqrt{P_n}-y_n| < \frac{j}{L}}}n^{\frac{3}{2}k}.
\end{align}
We therefore have the bounds
\begin{align*}
    \bigg(\frac{\ell x}{M}\bigg)^{\frac{3k}{2}} &\leq n^{\frac{3k}{2}} \leq \bigg(\frac{(\ell +1)x}{M}\bigg)^{\frac{3k}{2}} .
\end{align*}
Note that
\begin{align*}
    \bigg( \frac{(\ell + 1)x}{M}\bigg)^{\frac{3k}{2}} &= \bigg(\frac{\ell x}{M}\bigg)^{\frac{3k}{2}}\bigg(\frac{\ell + 1}{\ell}\bigg)^{\frac{3k}{2}} \\
    &= \bigg(\frac{\ell x}{M}\bigg)^{\frac{3k}{2}}\bigg(1+O_k\bigg(\frac{1}{\ell}\bigg)\bigg) \\
    &= \bigg(\frac{\ell x}{M}\bigg)^{\frac{3k}{2}} + O_k\bigg( \bigg(\frac{\ell x}{M}\bigg)^{\frac{3k}{2}}\frac{1}{\ell}\bigg). 
\end{align*}
Hence,
\begin{align}
    n^{\frac{3k}{2}} = \bigg(\frac{\ell x}{M}\bigg)^{\frac{3k}{2}} + O_k\bigg( \frac{\ell^{\frac{3k}{2}-1}x^{\frac{3k}{2}}}{M^{\frac{3k}{2}}}\bigg).
\end{align}
In summary, \eqref{S(j,ell)} can be rewritten as 
\begin{align}
    S_{j,\ell, M} &= \sum_{\substack{\frac{\ell x}{M} \leq n < \frac{(\ell +1)x}{M}\\\frac{j-1}{L} < |\sqrt{P_n} - y_n| \leq \frac{j}{L}}} \left(\bigg(\frac{\ell x}{M}\bigg)^{\frac{3k}{2}} +O\bigg(\frac{\ell^{\frac{3k}{2}-1}x^{\frac{3k}{2}}}{M^{\frac{3k}{2}}}\bigg)\right)\notag\\
    &= \sum_{\substack{\frac{\ell x}{M} \leq n < \frac{(\ell +1)x}{M}\\\frac{j-1}{L} < \{\sqrt{P_n}\} \leq \frac{j}{L}}} \bigg(\frac{\ell x}{M}\bigg)^{\frac{3k}{2}} + \sum_{\substack{\frac{\ell x}{M} \leq n < \frac{(\ell +1)x}{M}\\1-\frac{j}{L} < \{\sqrt{P_n}\} \leq 1-\frac{j-1}{L}}} \bigg(\frac{\ell x}{M}\bigg)^{\frac{3k}{2}} \notag\\
    &\quad+O\bigg(\sum_{\substack{\frac{\ell x}{M} \leq n < \frac{(\ell +1)x}{M}\\ \frac{1}{2}-x^{-\frac{3}{4}} < \{\sqrt{P_n}\} \leq \frac{1}{2}+x^{-\frac{3}{4}}}} \bigg(\frac{\ell x}{M}\bigg)^{\frac{3k}{2}} +\sum_{\substack{\frac{\ell x}{M} \leq n < \frac{(\ell +1)x}{M}\\\frac{j-1}{L} < |\sqrt{P_n} - y_n| \leq \frac{j}{L}}}\frac{\ell^{\frac{3k}{2}-1}x^{\frac{3k}{2}}}{M^{\frac{3k}{2}}}\bigg),\label{split Sjl sum}
\end{align}
where the second equality follows from \eqref{size of exceptional set}. We now focus on the first sum in \eqref{split Sjl sum}. 
\begin{align*}
    \sum_{\substack{\frac{\ell x}{M} \leq n < \frac{(\ell +1)x}{M}\\\frac{j-1}{L} < \{\sqrt{P_n}\} \leq \frac{j}{L}}} \bigg(\frac{\ell x}{M}\bigg)^{\frac{3k}{2}} &= \left(\frac{\ell x}{M}\right)^{\frac{3k}{2}} \# \left\{n \in \left[\frac{\ell x}{M}, \frac{(\ell + 1) x}{M}\right] \colon \{\sqrt{P_n}\} \in \left(\frac{j-1}{L}, \frac{j}{L}\right) \right\}\\
    &= \left(\frac{\ell x}{M}\right)^{\frac{3k}{2}} \left(\frac{x}{LM}+O(D(\mathcal{U}_{\ell,M}))\right),
\end{align*}
where
\begin{align}
    \mathcal{U}_{\ell,M} = \left\{ \{\sqrt{P_n}\} \colon n = \left\lfloor \frac{\ell x}{M}\right\rfloor + i, \ 1 \leq i \leq \left\lfloor\frac{x}{M}\right\rfloor\right\} \label{Family U definition}
\end{align}
and $D(\mathcal{U}_{\ell,M})$ represents the discrepancy of the family $\mathcal{U}_{\ell,M}$. Applying the same arguments to the second sum and error term in \eqref{split Sjl sum}, we have 
\begin{align}
    S_{j,\ell, M} &=\left(\frac{\ell x}{M}\right)^{\frac{3k}{2}} \cdot\frac{2x}{LM}+O\left(\left(\frac{\ell x}{M}\right)^{\frac{3k}{2}}D(\mathcal{U}_{\ell,M})+ \frac{\ell^{\frac{3k}{2}}x^{\frac{3k}{2} + \frac{1}{4}}}{M^{\frac{3k}{2}+1}} + \frac{\ell^{\frac{3k}{2} -1} x^{\frac{3k}{2} +1}}{LM^{\frac{3k}{2}+1}}\right).\label{split Sjl sum form 2}
\end{align}
To bound $D(\mathcal{U}_{\ell,M})$, we apply Lemma \ref{erdos turan inequality}. Then, for any $K \geq 1$,

\begin{align}
    D(\mathcal{U}_{\ell,M}) \leq \left\lfloor \frac{x}{M} \right\rfloor \frac{1}{K+1} + 3\sum_{m=1}^K\frac{1}{m}\bigg|\sum_{1 \leq i \leq \left\lfloor \frac{x}{M}\right\rfloor} e(m\sqrt{P_i})\bigg|. \label{erdos turan application}
\end{align}
Now our focus is to obtain an upper bound for the exponential sum in \eqref{erdos turan application}, where
\begin{align}
    \sum_{1\leq i \leq \left\lfloor\frac{x}{M}\right\rfloor} e(m\sqrt{P_i}) &= \sum_{\left\lfloor\frac{\ell x}{M}\right\rfloor \leq n \leq \left\lfloor \frac{(\ell+1)x}{M}\right\rfloor} e(m\sqrt{P_n}).
\end{align}
By Proposition \ref{sqrt pn u.d}, we obtain 
\begin{align*}
    &\bigg|\sum_{\left\lfloor \frac{\ell x}{M} \right\rfloor \leq n \leq \left\lfloor \frac{(\ell + 1) x}{M} \right\rfloor} e(m\sqrt{P_n})\bigg| \\
    &\leq \left(m\left|h'\left(\left\lfloor \frac{(\ell +1)x}{M} \right\rfloor \right)-h'\left(\left\lfloor \frac{\ell x}{M}\right\rfloor\right)\right| + 2\right)\bigg(4\left(mh''\left(\left\lfloor \frac{(\ell +1)x}{M} \right\rfloor \right)\right)^{-1/2}+3\bigg).
\end{align*}
Using Taylor series expansion, 
\begin{align*}
    \left|h'\left(\left\lfloor \frac{(\ell +1)x}{M} \right\rfloor \right)-h'\left(\left\lfloor \frac{\ell x}{M}\right\rfloor\right)\right| \ll \frac{x}{M}h''\left(\frac{\ell x}{M}\right).
\end{align*}
Thus, 
\begin{align}
    \bigg|\sum_{\left\lfloor \frac{\ell x}{M} \right\rfloor \leq n \leq \left\lfloor \frac{(\ell + 1) x}{M} \right\rfloor} e(m\sqrt{P_n})\bigg| &\ll \left( \frac{mx}{M}h''\left(\frac{\ell x}{M}\right)+ 2\right)\bigg(4\left(mh''\left(\left\lfloor \frac{(\ell +1)x}{M} \right\rfloor \right)\right)^{-1/2}+3\bigg)\notag\\
    &\ll \left( \frac{mx}{M}\cdot\frac{M^{\frac{1}{2}}}{\ell^{\frac{1}{2}}x^{\frac{1}{2}}}+ 2\right)\left(\frac{\ell^{\frac{1}{4}}x^{\frac{1}{4}}}{m^{\frac{1}{2}}M^{\frac{1}{4}}}+3\right)\notag\\
    &\ll \frac{m^{\frac{1}{2}}x^{\frac{3}{4}}}{\ell^{\frac{1}{4}}M^{\frac{3}{4}}}+\frac{mx^{\frac{1}{2}}}{M^{\frac{1}{2}}\ell^{\frac{1}{2}}}+\frac{\ell^{\frac{1}{4}}x^{\frac{1}{4}}}{m^{\frac{1}{2}}M^{\frac{1}{4}}}.\label{discrepancy sub-bound}
\end{align}
Substituting \eqref{discrepancy sub-bound} into \eqref{erdos turan application}, we obtain
\begin{align}
    D(\mathcal{U}_{\ell,M})&\ll \frac{x}{KM}+\sum_{m=1}^K \frac{x^{\frac{3}{4}}}{m^{\frac{1}{2}}\ell^{\frac{1}{4}}M^{\frac{3}{4}}}+\frac{x^{\frac{1}{2}}}{\ell^{\frac{1}{2}}M^{\frac{1}{2}}}+\frac{\ell^{\frac{1}{4}}x^{\frac{1}{4}}}{m^{\frac{3}{2}}M^{\frac{1}{4}}}\notag\\
    &\ll \frac{x}{KM} + \frac{K^{\frac{1}{2}}x^{\frac{3}{4}}}{\ell^{\frac{1}{4}}M^{\frac{3}{4}}}+\frac{Kx^{\frac{1}{2}}}{\ell^{\frac{1}{2}}M^{\frac{1}{2}}}+\frac{\ell^{\frac{1}{4}}x^{\frac{1}{4}}}{K^{\frac{1}{2}}M^{\frac{1}{4}}}.\label{discrepancy bound 2}
\end{align}

Combining \eqref{split Sjl sum form 2} and \eqref{discrepancy bound 2}, we see that 
\begin{align}
    S_{j,\ell, M}  
    &=
\frac{2 \ell^{\frac{3k}{2}} x^{\frac{3k}{2} + 1}}{LM^{\frac{3k}{2}+1}} + O\bigg(\frac{ \ell^{\frac{3k}{2}} x^{\frac{3k}{2} + 1}}{KM^{\frac{3k}{2}+1}} + \frac{ K^{\frac{1}{2}} \ell^{\frac{3k}{2} - \frac{1}{4}} x^{\frac{3k}{2} + \frac{3}{4}}}{M^{\frac{3k}{2} + \frac{3}{4}}} + \frac{K\ell^{\frac{3k}{2} - \frac{1}{2}}x^{\frac{3k}{2} + \frac{1}{2}}}{M^{\frac{3k}{2} + \frac{1}{2}}}+ \frac{  \ell^{\frac{3k}{2} + \frac{1}{4}} x^{\frac{3k}{2} + \frac{1}{4}}}{K^{\frac{1}{2}}M^{\frac{3k}{2} + \frac{1}{4}}}\notag\\
&\quad + \frac{\ell^{\frac{3k}{2}}x^{\frac{3k}{2} + \frac{1}{4}}}{M^{\frac{3k}{2}+1}} + \frac{\ell^{\frac{3k}{2} -1} x^{\frac{3k}{2} +1}}{LM^{\frac{3k}{2}+1}} \bigg).\notag\\\label{error term S(j,ell,M)}
\end{align}
Upon substitution of \eqref{error term S(j,ell,M)} back into \eqref{S j}, we show that
\begin{align}
    S_j  &=  \sum_{0\leq \ell \leq M-1} \bigg(\frac{2 \ell^{\frac{3k}{2}} x^{\frac{3k}{2} + 1}}{LM^{\frac{3k}{2}+1}} + O\bigg(\frac{ \ell^{\frac{3k}{2}} x^{\frac{3k}{2} + 1}}{KM^{\frac{3k}{2}+1}} + \frac{ K^{\frac{1}{2}} \ell^{\frac{3k}{2} - \frac{1}{4}} x^{\frac{3k}{2} + \frac{3}{4}}}{M^{\frac{3k}{2} + \frac{3}{4}}} + \frac{K\ell^{\frac{3k}{2} - \frac{1}{2}}x^{\frac{3k}{2} + \frac{1}{2}}}{M^{\frac{3k}{2} + \frac{1}{2}}} + \frac{  \ell^{\frac{3k}{2} + \frac{1}{4}} x^{\frac{3k}{2} + \frac{1}{4}}}{K^{\frac{1}{2}}M^{\frac{3k}{2} + \frac{1}{4}}}\notag\\
&\quad + \frac{\ell^{\frac{3k}{2}}x^{\frac{3k}{2} + \frac{1}{4}}}{M^{\frac{3k}{2}+1}} + \frac{\ell^{\frac{3k}{2} -1} x^{\frac{3k}{2} +1}}{LM^{\frac{3k}{2}+1}} \bigg)\bigg)\notag\\
&=\frac{2  x^{\frac{3k}{2} + 1}}{L(\frac{3}{2}k+1)}+O_k\bigg(\frac{  x^{\frac{3k}{2} + 1}}{K}+K^{\frac{1}{2}}  x^{\frac{3k}{2} + \frac{3}{4}}+Kx^{\frac{3k}{2}+\frac{1}{2}}+\frac{  M x^{\frac{3k}{2} + \frac{1}{4}}}{K^{\frac{1}{2}}}+x^{\frac{3k}{2}+\frac{1}{4}}+\frac{x^{\frac{3k}{2} + 1}}{LM}\bigg).\notag
\end{align}
Therefore, substituting the above back into the expression for $V_k(x,L)$ in \eqref{Vk(x,L)}, and applying Euler-Maclaurin summation to the sum over $j$, we obtain

\begin{align}
   V_k(x,L) &=\left( \frac{2}{L\sqrt{3}} \right)^{k}\sum_{j=1}^{L/2} j^k \bigg(\frac{2  x^{\frac{3k}{2} + 1}}{L(\frac{3}{2}k+1)}+O_k\bigg(\frac{  x^{\frac{3k}{2} + 1}}{K}+K^{\frac{1}{2}}  x^{\frac{3k}{2} + \frac{3}{4}}+Kx^{\frac{3k}{2}+\frac{1}{2}}\notag\\
   &\quad+ \frac{  M x^{\frac{3k}{2} + \frac{1}{4}}}{K^{\frac{1}{2}}}+x^{\frac{3k}{2}+\frac{1}{4}}+\frac{x^{\frac{3k}{2} + 1}}{LM}\bigg)\bigg)\notag\\
   &= \frac{2^{k}}{3^{k/2}L^{k}} \left(\frac{(\frac{L}{2})^{k+1}}{k+1}+\frac{(\frac{L}{2})^{k}}{2}+O(L^{k-1})\right) \bigg(\frac{2  x^{\frac{3k}{2} + 1}}{L(\frac{3}{2}k+1)}+O_k\bigg(\frac{  x^{\frac{3k}{2} + 1}}{K}+K^{\frac{1}{2}}  x^{\frac{3k}{2} + \frac{3}{4}}\notag\\
   &\quad+ Kx^{\frac{3k}{2}+\frac{1}{2}}+ \frac{  M x^{\frac{3k}{2} + \frac{1}{4}}}{K^{\frac{1}{2}}}+x^{\frac{3k}{2}+\frac{1}{4}}+\frac{x^{\frac{3k}{2} + 1}}{LM}\bigg)\bigg)\notag\\
   &=
\frac{x^{\frac{3k}{2} + 1}}{3^{k/2} (k+1) \left(\frac{3}{2} k + 1\right)} +  O_k\bigg(\frac{x^{\frac{3k}{2} + 1}}{ L } + 
\frac{Lx^{\frac{3k}{2} + 1}}{K} + LK^{\frac{1}{2}}x^{\frac{3k}{2} + \frac{3}{4}} + LKx^{\frac{3k}{2} + \frac{1}{2}}\notag\\
&\quad+ \frac{LMx^{\frac{3k}{2} + \frac{1}{4}}}{K^{\frac{1}{2}}} + Lx^{\frac{3k}{2} + \frac{1}{4}} + \frac{x^{\frac{3k}{2} + 1}}{M}\bigg).\label{Vk(x,L) final form}
\end{align}
Finally, we need to optimize $M, \ K$, and $L$ to minimize the error term in \eqref{Vk(x,L) final form}. To do this, apply Lemma \ref{lem:min_max_for_minor} with 
\[
F(M) =  \frac{x^{\frac{3k}{2} + 1}}{M}, G(M)=\frac{LMx^{\frac{3k}{2} + \frac{1}{4}}}{K^{\frac{1}{2}}} .
\]
We obtain $M = \frac{K^{1/4}x^{3/8}}{L^{1/2}}$. Substituting this back into \eqref{Vk(x,L) final form}, and applying Lemma \ref{lem:min_max_for_minor} with 
\begin{align*}
&F(L) =  \frac{x^{\frac{3k}{2} + 1}}{L}, G_1(L)=\frac{Lx^{\frac{3k}{2} + 1}}{K},G_2(L)=LK^{\frac{1}{2}}x^{\frac{3k}{2} + \frac{3}{4}}, G_3(L)=LKx^{\frac{3k}{2} + \frac{1}{2}}, \\
&G_4(L)=Lx^{\frac{3k}{2} + \frac{1}{4}} , G_5(L)=\frac{L^{1/2} x^{\frac{3k}{2} + \frac{5}{8}}}{K^{1/4}},
\end{align*}
we have 
\begin{align}
    V_{k}(x,L) &=
\frac{x^{\frac{3k}{2} + 1}}{3^{k/2} (k+1) \left(\frac{3}{2} k + 1\right)} +  O_k\bigg(
\frac{x^{\frac{3k}{2} + 1}}{K^{\frac{1}{2}}}
 + x^{\frac{3k}{2} + \frac{7}{8}}K^{\frac{1}{4}}
+ x^{\frac{3k}{2} + \frac{3}{4}}  K^{\frac{1}{2}}
\notag\\
&\quad + x^{\frac{3k}{2} + \frac{5}{8}}
 + \frac{x^{\frac{3k}{2} + \frac{3}{4}}}{K^{\frac{1}{6}}} \bigg).\label{equation 4.22}
\end{align}
Note that the last term can be absorbed by the third term in the error term of \eqref{equation 4.22}, so apply Lemma \ref{lem:min_max_for_minor} once again, we eventually get
\begin{align*}
    V_k(x,L) = \frac{x^{\frac{3k}{2} + 1}}{3^{k/2} (k+1) \left(\frac{3}{2} k + 1\right)} +  O_k\left(x^{\frac{3k}{2} + \frac{11}{12}}
\right).
\end{align*}
A similar calculation yields
\[
U_k(x,L) = \frac{x^{\frac{3}{2}k + 1}}{3^{k/2}(\frac{3}{2}k + 1)(k+1)} 
+O_k\left(x^{\frac{3k}{2} + \frac{11}{12}} \right).
\]
Thus, using \eqref{sandwich}, we have, for any $k\in\N$, 
\[
M_k(x) = \frac{x^{\frac{3}{2}k + 1}}{3^{k/2}(\frac{3}{2}k + 1)(k+1)} 
+O_k\left( x^{\frac{3k}{2} + \frac{11}{12}} \right).
\]
This finishes the proof of Theorem \ref{Moments theorem}. 
\end{proof}

 Theorem \ref{asymptotic for A(x) theorem} now follows as a corollary of Theorem \ref{Moments theorem} by taking $k=1$.

\end{document}